\documentclass[11pt]{amsart}
\usepackage[dvips]{graphicx}
\usepackage{amsmath}
\usepackage{amssymb}
\usepackage{amscd}
\usepackage{mathrsfs}
\usepackage{color}
 
\usepackage[all]{xy}%

\DeclareFontEncoding{OT2}{}{}
\DeclareFontSubstitution{OT2}{cmr}{m}{n}

\DeclareFontFamily{OT2}{cmr}{\hyphenchar\font45 }
\DeclareFontShape{OT2}{cmr}{m}{n}{%
   <5><6><7><8><9>gen*wncyr%
   <10><10.95><12><14.4><17.28><20.74><24.88>wncyr10}{}   
\DeclareFontShape{OT2}{cmr}{b}{n}{%
   <5><6><7><8><9>gen*wncyb%
   <10><10.95><12><14.4><17.28><20.74><24.88>wncyb10}{}

\DeclareMathAlphabet{\mathcyr}{OT2}{cmr}{m}{n}
\DeclareMathAlphabet{\mathcyb}{OT2}{cmr}{b}{n}
\SetMathAlphabet{\mathcyr}{bold}{OT2}{cmr}{b}{n}

\theoremstyle{plain}
    \newtheorem{theorem}{Theorem}[section]
    \newtheorem{proposition}[theorem]{Proposition}
    \newtheorem{lemma}[theorem]{Lemma}
    \newtheorem{corollary}[theorem]{Corollary}
\theoremstyle{definition}
    \newtheorem{definition}[theorem]{Definition}
    
    \newtheorem{example}[theorem]{Example}

\newcommand{\bZ}{\mathbb{Z}}
\newcommand{\bQ}{\mathbb{Q}}

\newcommand{\cA}{\mathcal{A}}
\newcommand{\cB}{\mathcal{B}}

\newcommand{\cL}{\mathcal{L}}

\newcommand{\bk}{\mathbf{k}}

\newcommand{\bp}{\boldsymbol{p}}

\newcommand{\Li}{\mathrm{Li}}

\newcommand{\OY}{\mathrm{OY}}

\newcommand{\wt}{\mathrm{wt}}

\newcommand{\blambda}{\boldsymbol\lambda}
\newcommand{\bmu}{\boldsymbol\mu}
\newcommand{\bnu}{\boldsymbol\nu}

\address{Department of Mathematics, Keio University, 3-14-1 Hiyoshi, Kouhoku-ku,Yokohama 223-8522, Japan}
\email{ono@math.keio.ac.jp}

\title{New functional equations of finite multiple polylogarithms}
\author{Masataka Ono}
\thanks{This research was supported in part by KAKENHI 26247004, as well as the JSPS Core-to-Core program ``Foundation of a Global Research Cooperative
Center in Mathematics focused on Number Theory and Geometry" and the KiPAS program 2013--2018 of the Faculty of Science and Technology at Keio University. }

\begin{document}

\maketitle
\begin{abstract}

We give a finite analogue of the well-known formula $\Li_{\underbrace{1, \ldots, 1}_n}(t)=\frac{1}{n!}\Li_1(t)^n$ of multiple polylogarithms for any positive integer $n$ by using the shuffle relation of finite multiple polylogarithms of Ono--Yamamoto type. Unlike the usual case, the terms regarded as error terms appear in this formula. As a corollary, we obtain $``t \leftrightarrow 1-t"$ type new functional equations of finite multiple polylogarithms of Ono--Yamamoto type and Sakugawa-Seki type.

\end{abstract}

\tableofcontents
\setcounter{section}{0}

\section{Introduction}

In this article, we give a finite analogue of the well-known formula $\Li_{\underbrace{1, \ldots, 1}_n}(t)=\frac{1}{n!}\Li_1(t)^n$ of multiple polylogarithms for any positive integer $n$ by using the shuffle relation of finite multiple polylogarithm (FMP) of Ono--Yamamoto type (OY-type). As a corollary, we obtain $``t \leftrightarrow 1-t"$ type new functional equations of FMPs of OY-type. In addition, by using the known relation between FMPs of OY-type and Sakugawa--Seki type (SS-type), we also obtain new functional equations of FMPs of SS-type, which seem to be difficult obtained only by using Sakugawa--Seki's results \cite{SS}.

First, we review of the history of FMPs. Recently, the author and Yamamoto \cite{OY} introduced finite multiple polylogarithms (FMPs) $\pounds^\OY_{\cA, \bk}(t)$ as an element of the $\bQ$-algebra $\cA_{\bZ[t]}:=\left.\bigl(\prod_p (\bZ/p\bZ)[t]\bigr)\right/\bigl(\bigoplus_p (\bZ/p\bZ)[t]\bigr)$, where $p$ runs through all the rational primes. Thus, an element of $\cA_{\bZ[t]}$ is represented by a family $(f_p)_p$ of polynomials $f_p \in (\bZ/p\bZ)[t]$, and two families $(f_p)_p$ and $(g_p)_p$ represent the same element of $\cA_{\bZ[t]}$ if and only if $f_p=g_p$ for all but finitely many primes $p$. We denote such an element of $\cA_{\bZ[t]}$ simply by $f_p$ omitting $(\ )_p$ if there is no fear of confusion. For example, we denote an element $(t^p \bmod{p})_p$ of $\cA_{\bZ[t]}$ by $t^{\boldsymbol{p}}$. 

Note that the idea considering several objects depending on a fixed prime $p$ in some adelic rings is due to Kaneko--Zagier's theory of finite multiple zeta values \cite{KZ}. Also note that the $\bQ$-algebra $\cA_{\bZ[t]}$ is denoted by $\cB$ in \cite{OY}. The symbol $\cA_{\bZ[t]}$ is Sakugawa--Seki's notation in \cite{SS}.

\begin{definition}[{\cite[Definition 1.2]{OY}}]
For a positive integer $r$ and an index $\bk=(k_1, \ldots, k_r) \in (\bZ_{\geq1})^r$, we define a finite multiple polylogarithm of Ono--Yamamoto type (OY-type for short) by 
$$
\pounds^{\OY}_{\cA, \bk}(t):=\sideset{}{'}\sum_{0<l_1, \ldots, l_r<p}\frac{t^{l_1+\cdots+l_r}}{l^{k_1}_1(l_1+l_2)^{k_2}\cdots(l_1+\cdots+l_r)^{k_r}} \bmod{p}
$$
as an element of $\cA_{\bZ[t]}$. Here, $\sideset{}{'}\sum$ denotes the sum of fractions whose denominators are prime to $p$.
\end{definition}

We respectively call integers $r$ and $k_1+\cdots+k_r$ the depth and weight of $\bk=(k_1, \ldots, k_r)$ and we denote the weight of index of $\bk$ by $\wt(\bk)$.

One of the reasons why we introduces FMPs of OY-type was to establish a finite analogue of the shuffle relation for multiple polylogarithms.

On the other hand, Sakugawa and Seki introduced another type of FMPs in \cite{SS}. We call their FMPs SS-type in this article. One of their motivations of introducing their FMPs was to establish functional equations of their FMPs.

More recently, Seki put a question about functional equations of FMPs of OY-type with the index $\{1\}^n:=(\underbrace{1, \ldots, 1}_n)$ for a positive integer $n$. The results of Kontsevich \cite[(A)]{K} and Elbaz-Vincent and Gangl \cite[PROPOSITION 5.9 (1)]{EG} say that $\pounds^\OY_{\cA, 1}(t)=\pounds^\OY_{\cA, 1}(1-t)$ holds. Futhermore, by using functional equations of FMPs of SS-type and the fact that FMPs of OY-type can be written in terms of FMPs of SS-type \cite[Proposition 3.26]{SS}, Seki \cite[Theorem 14.6]{Se} proved the equality $\pounds^\OY_{\cA, \{1\}^2}(t)=\pounds^\OY_{\cA, \{1\}^2}(1-t)$.

In this article, we give an answer to Seki's question, that is, we give functional equations between $\pounds^\OY_{\cA, \{1\}^n}(t)$ and $\pounds^\OY_{\cA, \{1\}^n}(1-t)$ for a positive integer $n$, which contain Seki's result as the case of $n=2$. In order to state our main theorem, we recall the definition of a variant $\zeta^{(i)}_\cA(\bk)$ of finite multiple zeta values (FMZVs) $\zeta_\cA(\bk)$.

\begin{definition}[{\cite[Definition 2.1]{OY}}] \label{var_FMZV}
For an index $\bk=(k_1, \ldots, k_r)$ and $1 \leq i \leq r$, we define a variant of FMZVs as an element of $\cA:=\left.\bigl(\prod_p \bZ/p\bZ\bigr)\right/\bigl(\bigoplus_p \bZ/p\bZ\bigr)$ by
$$
\zeta^{(i)}_\cA(\bk)
:=\sideset{}{'}\sum_{\substack{0<l_1, \ldots, l_r <p \\ (i-1)p< l_1+\cdots+l_r < ip}}
\frac{1}{l^{k_1}_1(l_1+l_2)^{k_2}\cdots (l_1+\cdots+l_r)^{k_r}}\bmod{p}.
$$
\end{definition}

Note that $\zeta^{(1)}_\cA(\bk)$ coincides with the usual FMZV $\zeta_\cA(\bk)$ defined by
$$
\zeta_\cA(\bk)=\sum_{0<n_1<\cdots<n_r<p}\frac{1}{n^{k_1}_1\cdots n^{k_r}_r} \bmod{p}
$$
and we see that $\zeta^{(r)}_\cA(\bk)=(-1)^{\wt(\bk)}\zeta_\cA(\bk)$.

The main theorem of this article is the following.

\begin{theorem}\label{main theorem}
For a positive integer $n$, we define two elements
\begin{equation}\label{f_n}
f_n(t):=\sum_{k=0}^{n-2}\left(\sum_{i=1}^{n-k-1}\zeta^{(i)}_\cA(\{1\}^{n-k-2}, 2)t^{i{\bp}}\right)\pounds^\OY_{\cA, \{1\}^k}(t)
\end{equation}
and
\begin{equation}\label{g_n}
g_n(t):=\sum_{k=0}^{n-2}\left(\sum_{i=1}^{n-k-2}\zeta^{(i)}_\cA(\{1\}^{n-k-2})t^{i{\bp}}\right)\pounds^\OY_{\cA, (2, \{1\}^k)}(t)
\end{equation}
of $\cA_{\bZ[t]}$. Here, we understand that these elements are equal to 0 if the sums are empty.
Then, we have
\begin{equation}\label{fin_analogue_of_shrel}
\pounds^\OY_{\cA, \{1\}^n}(t)=\frac{1}{n!}\pounds^\OY_{\cA, 1}(t)^n+\frac{1}{n!}\sum_{k=1}^n(k-1)!(f_k(t)+g_k(t))\pounds^\OY_{\cA, 1}(t)^{n-k}.
\end{equation}

\end{theorem}
We remark that this equality can be regarded as an finite analogue of the well-known formula $\Li_{\{1\}^n}(t)=\frac{1}{n!}\Li_1(t)^n$, where $\Li_{\bk}(t):=\sum_{0<n_1<\cdots<n_r}\frac{t^{n_r}}{n^{k_1}_1\cdots n^{k_r}_r}$ is the (one variable) multiple polylogarithm. 

As a corollary of our main theorem and the equality $\pounds^\OY_{\cA, 1}(t)=\pounds^\OY_{\cA, 1}(1-t)$, we obtain $``t \leftrightarrow 1-t"$ type functional equations of FMPs of OY-type.

\begin{corollary}\label{1-t}
For a positive integer $n$, set
$$
\cL_{\cA, n}(t):=\pounds^\OY_{\cA, \{1\}^n}(t)-\frac{1}{n!}\sum_{k=1}^n(k-1)!(f_k(t)+g_k(t))\pounds^\OY_{\cA, 1}(t)^{n-k}.
$$
Then we obtain 
$$
\cL_{\cA, n}(t)=\cL_{\cA, n}(1-t).
$$
\end{corollary}

\begin{proof}
By Theorem \ref{main theorem}, we have $\cL_{\cA, n}(t)=\frac{1}{n!}\pounds^\OY_{\cA, 1}(t)^n$. Therefore, the assertion holds by the result of Elbaz-Vincent and Gangl \cite[PROPOSITION 5.9 (1)]{EG}.
\end{proof}

The contents of this article is as follows. In Section 2, we prove Theorem \ref{main theorem} by using the shuffle relation of FMPs of OY-type. In Section 3, by using our main theorem, we give examples of functional equations of FMPs of OY-type. In the final section, by using the relation between FMPs of OY-type and SS-type, we also give functional equations of FMPs of SS-type, which seem to be difficult to obtain only using the results of \cite{SS}.

\section{Proof of the main theorem}

In this section, we prove the main theorem by using the shuffle relation of FMPs of OY-type, which was proved by the author and S. Yamamoto in \cite{OY}.

First, we explicitly calculate the shuffle relation of $\pounds^\OY_{\cA, \{1\}^{n-1}}(t)$ and $\pounds^\OY_{\cA, 1}(t)$ for a positive integer $n$.

\begin{lemma}\label{sh prod}
For a positive integer $n$,  we have
\begin{equation*}
\pounds^\OY_{\cA,\{1\}^{n-1}}(t)\pounds^\OY_{\cA,1}(t)=n\pounds^\OY_{\cA, \{1\}^n}(t)-f_n(t)-g_n(t).
\end{equation*}
Recall the definitions of $f_n(t)$ and $g_n(t)$. See \eqref{f_n} and \eqref{g_n}.
\end{lemma}

\begin{proof}
For $0 \leq k \leq n-1$, set
$$
F_k(t):=\pounds^\OY_{\cA}(\{1\}^{n-k-1}, (1), \{1\}^k; t).
$$
Here, for indices $\blambda = (\lambda_1, \ldots, \lambda_a), \bmu =(\mu_1, \ldots, \mu_b)$ and $\bnu=(\nu_1, \ldots, \nu_c)$ ($a, b, c \in \bZ_{\geq0}$), 
$\pounds^\OY_{\cA}(\blambda, \bmu, \bnu; t)$ is the FMP of type $(\blambda, \bmu, \bnu)$ \cite[Definition 3.1]{OY} defined by 
$$
\pounds^\OY_{\cA}(\blambda, \bmu, \bnu; t)
:=\sideset{}{'}\sum_{\substack{0<l_1, \ldots, l_a <p \\ 0<m_1, \ldots, m_b <p \\0<n_1, \ldots, n_c <p}}
\frac{t^{L_a+M_b+N_c}}{
\displaystyle\prod_{x=1}^a L^{\lambda_x}_x
\displaystyle\prod_{y=1}^b M^{\mu_y}_y
\displaystyle\prod_{z=1}^c \left(L_a+M_b+N_z\right)^{\nu_z}
}
\in \cA_{\bZ[t]},
$$
where $L_x:=l_1+\cdots+l_x, M_y:=m_1+\cdots+m_y$ and $N_z:=n_1+\cdots+n_z$.
By \cite[Remark 3.2]{OY}, we have $F_0(t)=\pounds^\OY_{\cA,\{1\}^{n-1}}(t)\pounds^\OY_{\cA,1}(t)$ and $F_{n-1}(t)=\pounds^\OY_{\cA, \{1\}^n}(t)$.
By using \cite[Proposition 3.7]{OY} in the case $\blambda:=\{1\}^{n-k-1}, \bmu:=(1)$ and $\bnu:=\{1\}^k$ $(0\leq k \leq n-2)$, we obtain
\begin{multline}\label{eq:reccurence}
F_k(t)=F_{k+1}(t)+\pounds^\OY_{\cA, \{1\}^n}(t)-\left(\sum_{i=1}^{n-k-1}\zeta^{(i)}_\cA(\{1\}^{n-k-2}, 2)t^{i\bp}\right)\pounds^\OY_{\cA, \{1\}^k}(t)\\
            -\left(\sum_{i=1}^{n-k-2}\zeta^{(i)}_\cA(\{1\}^{n-k-2})t^{i\bp}\right)\pounds^\OY_{\cA,(2, \{1\}^k)}(t).
\end{multline}
Therefore, the statement holds from taking the telescoping sum of \eqref{eq:reccurence}.
\end{proof}

\begin{proof}[Proof of Theorem \ref{main theorem}]
We prove the statement by the induction on $n \geq1$.
Note that the statement for $n=1$ holds by \cite[(A)]{K} or \cite[PROPOSITION 5.9]{EG}. 
For $n \geq 2$, assume that the statement holds for $n-1$:
\begin{equation}\label{eq:r-1}
\pounds^\OY_{\cA, \{1\}^{n-1}}(t)=\frac{1}{(n-1)!}\pounds^\OY_{\cA, 1}(t)^{n-1}+\frac{1}{(n-1)!}\sum_{k=1}^{n-1}(k-1)!(f_k(t)+g_k(t))\pounds^\OY_{\cA, 1}(t)^{n-k-1}.
\end{equation}
By Lemma \ref{sh prod}, the product of the left hand side of \eqref{eq:r-1} and $\pounds^\OY_{\cA,1}(t)$ coincides with
\begin{equation}\label{eq:1*n-1}
\pounds^\OY_{\cA,\{1\}^{n-1}}(t)\pounds^\OY_{\cA,1}(t)=n\pounds^\OY_{\cA, \{1\}^n}(t)-f_n(t)-g_n(t).
\end{equation}
On the other hand, the product of the right hand side of \eqref{eq:r-1} and $\pounds^\OY_{\cA,1}(t)$ coincides with
 
\begin{multline}\label{eq:n-1*1}
  \frac{1}{(n-1)!}\pounds^\OY_{\cA,1}(t)^n+\frac{1}{(n-1)!}\sum_{k=1}^{n-1}(k-1)!(f_k(t)+g_k(t))\pounds^\OY_{\cA,1}(t)^{n-k}\\
=n\left(\frac{1}{n!}\pounds^\OY_{\cA, 1}(t)^n+\frac{1}{n!}\sum_{k=1}^{n-1}(k-1)!(f_k(t)+g_k(t))\pounds^\OY_{\cA, 1}(t)^{n-k}\right).
\end{multline}
Therefore, by \eqref{eq:1*n-1} and \eqref{eq:n-1*1}, we see that the statement holds for $n$.
\end{proof}

\section{Functional equations of finite multiple polylogarithms of Ono-Yamamoto type}

Next, we give examples of our main theorem for $n=2, 3$ and $4$. We use the following lemmas for calculating our examples.

\begin{lemma}[{\cite[Theorem 4.3]{H}}]\label{repitition}
For any positive integer $k$ and $r$, we have
$\zeta_\cA(\underbrace{k, \ldots, k}_r)=0$. In particular, we have $\zeta_\cA(k)=0$ for any positive integer $k$.
\end{lemma}

\begin{lemma}[{\cite[Table 2]{Sa}}, for example]\label{weight 4}
For any index $\bk$ of weight 4, we have $\zeta_\cA(\bk)=0$.
\end{lemma}

\begin{example}\label{eg}
\begin{enumerate}
\item
First, we consider the case $n=2$. By the definition of $\zeta^{(i)}_\cA$ and Lemma \ref{repitition}, we obtain $f_2(t)=\zeta_\cA(2)t^{\bp}=0$. By the definition of $g_2(t)$, we have $g_2(t)=0$. Thus Theorem \ref{main theorem} says the equality $\pounds^\OY_{\cA, (1,1)}(t)=\frac{1}{2}\pounds^\OY_{\cA, 1}(t)^2$. Moreover, by Corollary \ref{1-t}, we obtain Seki's result $\pounds^\OY_{\cA, (1,1)}(t)=\pounds^\OY_{\cA, (1,1)}(1-t)$. 

\item Next, we consider the case of $n=3$. By an easy calculation and $\zeta^{(2)}_\cA(1,2)=-\zeta_\cA(1,2)$ \cite[Remark 2.2]{OY}, $f_3(t)$ can be calculated as follows.
\begin{align}\label{f_3}
f_3(t)=\zeta^{(1)}_\cA(1,2)t^{\bp}+\zeta^{(2)}_\cA(1,2)t^{2\bp}
        =\zeta_\cA(1,2)t^{\bp}(1-t^{\bp})=\zeta_\cA(1,2)t^{p}(1-t)^{\bp}.
\end{align}
The last equality holds since $1-t^p=(1-t)^p$ holds in $(\bZ/p\bZ)[t]$ for all primes $p$. On the other hand, we have $g_3(t)=\zeta_\cA(1)t^{\bp}\pounds^\OY_{\cA, 2}(t)=0$ by Lemma \ref{repitition}. Therefore, Theorem \ref{main theorem} says the equality
$$
\pounds^\OY_{\cA, \{1\}^3}(t)= \frac{1}{3!}\pounds^\OY_{\cA, 1}(t)^3+\frac{1}{3}\zeta_\cA(1,2)t^{\bp}(1-t)^{\bp}.
$$
Moreover, since $f_3(t)=f_3(1-t)$ by \eqref{f_3}, we see that Corollary \ref{1-t} says the equality $\pounds^\OY_{\cA, \{1\}^3}(t)=\pounds^\OY_{\cA, \{1\}^3}(1-t)$. 

\item Furthermore, we consider the case of $n=4$. By an easy calculation, we obtain
\begin{align}\label{f_4}
f_4(t)=&\zeta_\cA(1,1,2)t^{\bp}+\zeta^{(2)}_\cA(1,1,2)t^{2{\bp}}+\zeta^{(3)}_\cA(1,1,2)t^{3{\bp}}\nonumber\\
           &+(\zeta_\cA(1,2)t^{\bp}+\zeta^{(2)}_\cA(1,2)t^{2\bp})\pounds^\OY_{\cA, 1}(t)+\zeta_\cA(2)t^{\bp}\pounds^\OY_{\cA, (1,1)}(t)\nonumber\\
        =&\zeta_\cA(1,1,2)t^{\bp}(1-t^{2\bp})+\zeta^{(2)}_\cA(1,1,2)t^{2\bp}+f_3(t)\pounds^\OY_{\cA, 1}(t).
\end{align}
By \cite[Example 2.6, (ii)]{OY}, $\zeta^{(2)}_\cA(1,1,2)$ is a sum of FMZVs of weight 4, we see that $f_4(t)=f_3(t)\pounds^\OY_{\cA, 1}(t)$ by Lemma \ref{weight 4}. 
On the other hand, since $\zeta_\cA(1,1)=\zeta^{(2)}_\cA(1,1)=0$ by Lemma \ref{repitition}, we have
$$
g_4(t)=(\zeta_\cA(1,1)t^{\bp}+\zeta^{(2)}_\cA(1,1)t^{2\bp})\pounds^\OY_{\cA, 2}(t)+\zeta_\cA(1)t^{\bp}\pounds^\OY_{\cA, (2,1)}(t)=0.
$$
Therefore, we obtain
\begin{align*}
\pounds^\OY_{\cA, \{1\}^4}(t)=\frac{1}{24}\pounds^\OY_{\cA, 1}(t)^4+\frac{1}{3}\zeta_\cA(1,2)t^{\bp}(1-t)^{\bp}\pounds^\OY_{\cA, 1}(t).
\end{align*}
Moreover, since $f_4(t)=f_4(1-t)$ by \eqref{f_4}, we have $\pounds^\OY_{\cA, \{1\}^4}(t)=\pounds^\OY_{\cA, \{1\}^4}(1-t)$.

\item
Finally, consider the case of $n=5$. In this case, it is difficult to expect the $"t \leftrightarrow 1-t"$ type relation of $\pounds^\OY_{\cA, \{1\}^n}(t)$ for $n \geq5$. 

First, by \cite[Example 2.6]{OY} and Lemma \ref{repitition}, we have $\zeta_\cA(1,1,1)$$=-\zeta^{(3)}_\cA(1,1,1)=0$ and $\zeta^{(2)}_\cA(1,1,1)=4\zeta_\cA(1,1,1)+\zeta_\cA(2,1)+\zeta_\cA(1,2)=0$. Therefore, we obtain
\begin{align*}
g_5(t)=&(\zeta_\cA(1,1,1)t^{\bp}+\zeta^{(2)}_\cA(1,1,1)t^{2\bp}+\zeta^{(3)}_\cA(1,1,1)t^{3\bp})\pounds^\OY_{\cA, 2}(t)\\
            &+(\zeta_\cA(1,1)t^{\bp}+\zeta^{(2)}_\cA(1,1)t^{2\bp})\pounds^\OY_{\cA, (2,1)}(t)+\zeta_\cA(1)t^{\bp}\pounds^\OY_{\cA, (2,1,1)}(t)
          =0.
\end{align*}

Thus, Theorem \ref{main theorem} says that
\begin{align*}
\pounds^\OY_{\cA, \{1\}^5}(t)=&\frac{1}{5!}\pounds^\OY_{\cA, 1}(t)^5+\frac{2!f_3(t)\pounds^\OY_{\cA, 1}(t)^2+3!f_4(t)\pounds^\OY_{\cA, 1}(t)+4!f_5(t)}{5!}\\
                                              =&\frac{1}{5!}\pounds^\OY_{\cA, 1}(t)^5+\frac{1}{15}f_3(t)\pounds^\OY_{\cA, 1}(t)^2+\frac{1}{5}f_5(t).
\end{align*}
Since $f_3(t)\pounds^\OY_{\cA, 1}(t)^2=f_3(1-t)\pounds^\OY_{\cA, 1}(1-t)^2$, we have
$$
\pounds^\OY_{\cA, \{1\}^5}(t)-\frac{1}{5}f_5(t)=\pounds^\OY_{\cA, \{1\}^5}(1-t)-\frac{1}{5}f_5(1-t).
$$
Next, we have
\begin{align*}
f_5(t)=&\zeta^{(1)}_\cA(1,1,1,2)t^{\bp}+\zeta^{(2)}_\cA(1,1,1,2)t^{2\bp}+\zeta^{(3)}_\cA(1,1,1,2)t^{3\bp}+\zeta^{(4)}_\cA(1,1,1,2)t^{4\bp}\\
           &+\biggl(\zeta^{(1)}_\cA(1,1,2)t^{\bp}+\zeta^{(2)}_\cA(1,1,2)t^{2\bp}+\zeta^{(3)}_\cA(1,1,2)t^{3\bp}\biggr)\pounds^\OY_{\cA, 1}(t)\\
        =&\zeta_\cA(1,1,1,2)(t^{\bp}-t^{4\bp})+\zeta^{(2)}_\cA(1,1,1,2)(t^{2\bp}-t^{3\bp})+f_3(t)\pounds^\OY_{\cA, 1}(t)^2\\
        =&\zeta_\cA(1,1,1,2)t^{\bp}(1-t^{\bp})(1+t^{\bp}+t^{2\bp})+\zeta^{(2)}_\cA(1,1,1,2)t^{2\bp}(1-t^{\bp})+f_3(t)\pounds^\OY_{\cA, 1}(t)^2.
\end{align*} 
Therefore, we see that
\begin{align*}
\pounds^\OY_{\cA, \{1\}^5}(t)-\pounds^\OY_{\cA, \{1\}^5}(1-t)
=&\frac{f_5(t)-f_5(1-t)}{5}\\
=&\frac{2\zeta_\cA(1,1,1,2)+\zeta^{(2)}_\cA(1,1,1,2)}{10}t^{\bp}(1-t^{\bp})(2t^{\bp}-1).
\end{align*}
By \cite[Example 2.6 (2)]{OY} and \cite[Table 2]{Sa}, we see that $\zeta_\cA(1,1,1,2)=B_{\bp-5}$ and $\zeta^{(2)}_\cA(1,1,1,2)=0$. Here, we set $B_{\bp-5}:=(B_{p-5} \bmod{p})_p \in \cA$. Therefore, we obtain
$$
\pounds^\OY_{\cA, \{1\}^5}(t)-\pounds^\OY_{\cA, \{1\}^5}(1-t)=\frac{B_{\bp-5}}{5}t^{\bp}(1-t^{\bp})(2t^{\bp}-1).
$$
Thus, since it is conjectured that $B_{\bp-5}$ does not vanish in $\cA$ (for example, see \cite[Conjecture 2.1]{Z}), we see that $\pounds^\OY_{\cA, \{1\}^5}(t) \neq \pounds^\OY_{\cA, \{1\}^5}(1-t)$.

\end{enumerate}
\end{example}

\section{Functional equations of finite multiple polylogarithms of Sakugawa-Seki type}

We end this article with new functional equations of FMPs of SS-type. 

\begin{definition}[{\cite[Definition 3.8]{SS}}]
Let $r$ be a positive integer and $\bk=(k_1, \ldots, k_r)$ an index. Then we define finite harmonic multiple polylogarithms and 1-variable finite multiple polylogarithms as follows:
$$
\pounds^{*}_{\cA, \bk}(t_1, \ldots, t_r)=\sum_{0<n_1<\cdots<n_r<p}\frac{t^{n_1}_1\cdots t^{n_r}_r}{n^{k_1}_1\cdots n^{k_r}_r }\bmod{p} \in \cA_{\bZ[\boldsymbol{t}]},
$$
$$
\pounds_{\cA, \bk}(t):=\pounds^*_{\cA, \bk}(\{1\}^{r-1}, t) \in \cA_{\bZ[t]},\quad
\widetilde{\pounds}_{\cA, \bk}(t):=\pounds^*_{\cA, \bk}(t, \{1\}^{r-1}) \in \cA_{\bZ[t]}.
$$
Here, for an $r$-tuple of variables $\boldsymbol{t}:=(t_1, \ldots, t_r)$, we set $\cA_{\bZ[\boldsymbol{t}]}:=\left.\bigl(\prod_p (\bZ/p\bZ)[\boldsymbol{t}]\bigr)\right/\bigl(\bigoplus_p (\bZ/p\bZ)[\boldsymbol{t}]\bigr)$. \end{definition}

Now we prepare the following notation to describe the relation between the FMPs of OY-type and SS-type (cf. \cite[Section 2]{OY}). First, for a positive integer $r$, set

$$
[r]:=\{1, 2, \ldots, r\}
$$
and
$$
\Phi_r:=\bigsqcup_{s=1}^r\Phi_{r,s},\quad \Phi_{r,s}:=\{\phi : [r]\rightarrow [s] : \text{surjective}\mid \phi(a)\neq \phi(a+1) \text{ for all } a \in [r-1]\}.
$$
Next, for $\phi \in \Phi_{r,s}$, set $s_\phi:=s$. Furthermore, for $\phi \in \Phi_r$ and $1\leq i \leq r$, we define an integer $\delta_\phi(i)$ by
\begin{equation*} \label{eq:delta}
\delta_\phi(i):=\#\{a \in [i-1] \mid \phi(a)>\phi(a+1)\} \quad (1\leq i \leq r).
\end{equation*}
Finally, a map $\beta : \Phi_r \rightarrow [r]$ is defined by $\beta(\phi):=\delta_{\phi}(r)+1$ and we set 
$$
\Phi^i_r :=\beta^{-1}(i).
$$

\begin{proposition}[{\cite[Proposition 3.26]{SS}}]\label{OYtoSS}
For an index $\bk=(k_1, \ldots, k_r)$, we have
$$
\pounds^\OY_{\cA, \bk}(t)=\sum_{i=1}^rt^{(i-1){p}}\sum_{\phi \in \Phi^i_r}\pounds^*_{\cA, (\sum_{\phi(j)=1}k_j, \ldots, \sum_{\phi(j)=s_\phi}k_j)}(\{1\}^{\phi(r)-1}, t, \{1\}^{s_\phi-\phi(r)}).
$$
\end{proposition}
By Proposition \ref{OYtoSS}, our main theorem gives functional equations of FMPs of SS-type. It seems very difficult to obtain our functional equations of FMPs of SS-type only using Sakugawa-Seki's theory \cite{SS} and without using the shuffle relation of FMPs of OY-type. We describe only two functional equations of FMPs of SS-type which are obtained from that of FMPs of OY-type with indices $\{1\}^3$ and $\{1\}^4$.

\begin{corollary}
We have
\begin{align*}
&(1+t^{\bp})\pounds_{\cA, \{1\}^3}(t)+t^{\bp}(1+t^{\bp})\widetilde{\pounds}_{\cA,\{1\}^3}(t)
+2t^{\bp}\pounds^*_{\cA, \{1\}^3}(1,t,1)+t^{\bp}\pounds_{\cA, (2,1)}(t)+t^{\bp}\widetilde{\pounds}_{\cA, (1,2)}(t)\\
=&(2-t^{\bp})\pounds_{\cA, \{1\}^3}(1-t)+(2-t^{\bp})(1-t^{\bp})\widetilde{\pounds}_{\cA,\{1\}^3}(1-t)
+2(1-t^{\bp})\pounds^*_{\cA, \{1\}^3}(1,1-t,1)\\
&+(1-t^{\bp})\pounds_{\cA, (1,2)}(1-t)+(1-t^{\bp})\widetilde{\pounds}_{\cA, (2,1)}(1-t).
\end{align*}
\end{corollary}

\begin{corollary}
We have
\begin{align*}
  &(1+4t^{\bp}+t^{2\bp})\pounds_{\cA, \{1\}^4}(t)+2t^{\bp}(2+t^{\bp})\pounds^*_{\cA, \{1\}^4}(1,1,t,1)+2t^{\bp}(1+2t^{\bp})\pounds^*_{\cA, \{1\}^4}(1,t,1,1)\\
  &+t^{\bp}(1+4t^{\bp}+t^{2\bp})\widetilde{\pounds}_{\cA, \{1\}^4}(t)+t^{\bp}(1+3t^{\bp})\widetilde{\pounds}_{\cA, (2,1,1)}(t)+2t^{\bp}(1+t^{\bp})\pounds^*_{\cA,(1,2,1)}(1,t,1)\\
  &+t^{\bp}(3+t^{\bp})\pounds_{\cA,(1,1,2)}(t)+t^{\bp}\biggl(\pounds_{\cA,(2,1,1)}(t)+\pounds^*_{\cA,(2,1,1)}(1,t,1)+\pounds_{\cA,(1,2,1)}(t)+\pounds_{\cA,(2,2)}(t)\biggr)\\
  &+t^{2\bp}\biggl(\widetilde{\pounds}_{\cA,(1,1,2)}(t)+\pounds^*_{\cA,(1,1,2)}(1,t,1)+\widetilde{\pounds}_{\cA, (1,2,1)}(t)+\widetilde{\pounds}_{\cA,(2,2)}(t)\biggr)\\
= &(6-6t^{\bp}+t^{2\bp})\pounds_{\cA, \{1\}^4}(1-t)+2(1-t^{\bp})(3-t^{\bp})\pounds^*_{\cA, \{1\}^4}(1,1,1-t,1)\\
&+2(1-t^{\bp})(3-2t^{\bp})\pounds^*_{\cA, \{1\}^4}(1,1-t,1,1)+(1-t^{\bp})(6-6t^{\bp}+t^{2\bp})\widetilde{\pounds}_{\cA, \{1\}^4}(1-t)\\
&+(1-t^{\bp})(4-3t^{\bp})\widetilde{\pounds}_{\cA, (2,1,1)}(1-t)+2(1-t^{\bp})(2-t^{\bp})\pounds^*_{\cA,(1,2,1)}(1,1-t,1)\\
  &+(1-t^{\bp})(4-t^{\bp})\pounds_{\cA,(1,1,2)}(1-t)+(1-t^{\bp})\biggl(\pounds_{\cA,(2,1,1)}(1-t)+\pounds^*_{\cA,(2,1,1)}(1,1-t,1)\\
  &+\pounds_{\cA,(1,2,1)}(1-t)+\pounds_{\cA,(2,2)}(1-t)\biggr)+(1-t^{\bp})^2\biggl(\widetilde{\pounds}_{\cA,(1,1,2)}(1-t)+\pounds^*_{\cA,(1,1,2)}(1,1-t,1)\\
  &+\widetilde{\pounds}_{\cA, (1,2,1)}(1-t)+\widetilde{\pounds}_{\cA,(2,2)}(1-t)\biggr).
\end{align*}
\end{corollary}

\section*{Acknowledgement}

The author expresses his sincere gratitude to Dr. Shin-ichiro Seki for introducing me a question concerning the functional equations of FMPs of OY-type. He would like to thank Dr. Kenji Sakugawa and Dr. Shin-ichiro Seki for their valuable comments and helpful discussions at Keio University. He also would like to thank Prof. Kenichi Bannai, Prof. Shuji Yamamoto, Dr. Kenji Sakugawa and Dr. Shin-ichiro Seki for careful reading of the manuscript.

\end{document}